\documentclass[12pt]{amsart}
\usepackage{amsmath,amssymb,amsbsy,amsfonts,amsthm,latexsym, amsopn,amstext,amsxtra,euscript,amscd,mathrsfs,color,bm, cite,multirow,tabularx}
\usepackage{todonotes,dirtytalk}   
\usepackage{url}
\usepackage[colorlinks,linkcolor=blue,anchorcolor=blue,citecolor=blue]{hyperref}
\usepackage{color}
\usepackage{comment}
\usepackage{dirtytalk}

\begin{document}
\newtheorem{problem}{Problem}
\newtheorem{theorem}{Theorem}
\newtheorem{lemma}[theorem]{Lemma}
\newtheorem{crit}[theorem]{Criterion}
\newtheorem{claim}[theorem]{Claim}
\newtheorem{cor}[theorem]{Corollary}
\newtheorem{prop}[theorem]{Proposition}
\newtheorem{definition}{Definition}
\newtheorem{question}[theorem]{Question}
\newtheorem{rem}[theorem]{Remark}
\newtheorem{Note}[theorem]{Notation}


\def\cA{{\mathcal A}}
\def\cB{{\mathcal B}}
\def\cC{{\mathcal C}}
\def\cD{{\mathcal D}}
\def\cE{{\mathcal E}}
\def\cF{{\mathcal F}}
\def\cG{{\mathcal G}}
\def\cH{{\mathcal H}}
\def\cI{{\mathcal I}}
\def\cJ{{\mathcal J}}
\def\cK{{\mathcal K}}
\def\cL{{\mathcal L}}
\def\cM{{\mathcal M}}
\def\cN{{\mathcal N}}
\def\cO{{\mathcal O}}
\def\cP{{\mathcal P}}
\def\cQ{{\mathcal Q}}
\def\cR{{\mathcal R}}
\def\cS{{\mathcal S}}
\def\cT{{\mathcal T}}
\def\cU{{\mathcal U}}
\def\cV{{\mathcal V}}
\def\cW{{\mathcal W}}
\def\cX{{\mathcal X}}
\def\cY{{\mathcal Y}}
\def\cZ{{\mathcal Z}}

\def\A{{\mathbb A}}
\def\B{{\mathbb B}}
\def\C{{\mathbb C}}
\def\D{{\mathbb D}}
\def\E{{\mathbb E}}
\def\F{{\mathbb F}}
\def\G{{\mathbb G}}
\def\I{{\mathbb I}}
\def\J{{\mathbb J}}
\def\K{{\mathbb K}}
\def\L{{\mathbb L}}
\def\M{{\mathbb M}}
\def\N{{\mathbb N}}
\def\O{{\mathbb O}}
\def\P{{\mathbb P}}
\def\Q{{\mathbb Q}}
\def\R{{\mathbb R}}
\def\S{{\mathbb S}}
\def\T{{\mathbb T}}
\def\U{{\mathbb U}}
\def\V{{\mathbb V}}
\def\W{{\mathbb W}}
\def\X{{\mathbb X}}
\def\Y{{\mathbb Y}}
\def\Z{{\mathbb Z}}

\def\ep{{\mathbf{e}}_p}
\def\eq{{\mathbf{e}}_q}
\def\cal#1{\mathcal{#1}}

\def\scr{\scriptstyle}
\def\\{\cr}
\def\({\left(}
\def\){\right)}
\def\[{\left[}
\def\]{\right]}
\def\<{\langle}
\def\>{\rangle}
\def\fl#1{\left\lfloor#1\right\rfloor}
\def\rf#1{\left\lceil#1\right\rceil}
\def\le{\leqslant}
\def\ge{\geqslant}
\def\eps{\varepsilon}
\def\mand{\qquad\mbox{and}\qquad}

\def\sssum{\mathop{\sum\ \sum\ \sum}}
\def\ssum{\mathop{\sum\, \sum}}
\def\ssumw{\mathop{\sum\qquad \sum}}

\def\vec#1{\mathbf{#1}}
\def\inv#1{\overline{#1}}
\def\num#1{\mathrm{num}(#1)}
\def\dist{\mathrm{dist}}

\def\fA{{\mathfrak A}}
\def\fB{{\mathfrak B}}
\def\fC{{\mathfrak C}}
\def\fU{{\mathfrak U}}
\def\fV{{\mathfrak V}}

\newcommand{\bflambda}{{\boldsymbol{\lambda}}}
\newcommand{\bfxi}{{\boldsymbol{\xi}}}
\newcommand{\bfrho}{{\boldsymbol{\rho}}}
\newcommand{\bfnu}{{\boldsymbol{\nu}}}

\def\GL{\mathrm{GL}}
\def\SL{\mathrm{SL}}

\def\Hba{\overline{\cH}_{a,m}}
\def\Hta{\widetilde{\cH}_{a,m}}
\def\Hb1{\overline{\cH}_{m}}
\def\Ht1{\widetilde{\cH}_{m}}

\def\flp#1{{\left\langle#1\right\rangle}_p}
\def\flm#1{{\left\langle#1\right\rangle}_m}

\def\Zm{\Z/m\Z}

\def\Err{{\mathbf{E}}}
\def\O{\mathcal{O}}

\def\cc#1{\textcolor{red}{#1}}
\newcommand{\commT}[2][]{\todo[#1,color=green!60]{Tim: #2}}
\newcommand{\commB}[2][]{\todo[#1,color=red!60]{Bryce: #2}}

\newcommand{\comm}[1]{\marginpar{%
\vskip-\baselineskip 
\raggedright\footnotesize
\itshape\hrule\smallskip#1\par\smallskip\hrule}}
\newcolumntype{L}{>{\raggedright\arraybackslash}X}

\def\xxx{\vskip5pt\hrule\vskip5pt}

\def\dmod#1{\,\left(\textnormal{mod }{#1}\right)}


\title{The modified prime sieve for primitive elements in finite fields}

\author[G.K. Bagger]{Gustav Kj\ae rbye Bagger}
\address{School of Science, The University of New South Wales Canberra, Australia}
\email{g.bagger@unsw.edu.au}

\author[J. Punch]{James Punch}
\address{School of Science, The University of New South Wales Canberra, Australia}
\email{j.punch@unsw.edu.au}
 
\date{\today}
\pagenumbering{arabic}


\begin{abstract}
Let $r \ge 2$ be an integer, $q$ a prime power and $\F_{q}$ the finite field with $q$ elements. Consider the problem of showing existence of primitive elements in a subset $\cal{A} \subseteq \F_{q^r}$. We prove a sieve criterion for existence of such elements, dependent only on an estimate for the character sum $\sum_{\gamma \in \cal{A}}\chi(\gamma)$. The flexibility and direct applicability of our criterion should be of considerable interest for problems in this field. We demonstrate the utility of our result by tackling a problem of Fernandes and Reis~\cite{FernandesReis2021} with $\cal{A}$ avoiding affine hyperplanes, obtaining significant improvements over previous knowledge. 
\end{abstract}

\maketitle
\let\thefootnote\relax
\footnote{\textit{Affiliation}: School of Science, The University of New South Wales Canberra, Australia.}
\footnote{\textit{Corresponding author}: James Punch (j.punch@unsw.edu.au).}
\footnote{\textit{Key phrases}: Primitive elements, Character sum estimates, Prime sieves}
\section{Introduction}
Let $\F_{q^r}$ be the finite field with $q^r$ elements for some prime power $q$. An element $\gamma \in \F_{q^r}$ is called {\it primitive} if it generates $\F_{q^r}^*$, a cyclic group of order $q^r-1$. For each divisor $e\mid q^r-1$, call an element $\gamma \in \F_{q^r}^*$ {\it e-free} if there is no $\beta \in \F_{q^r}^*$ satisfying $\gamma = \beta^d$ whenever $d \mid e$ and $d\neq 1$. It follows that $\gamma \in F_{q^r}^*$ is primitive if and only if it is $(q^r-1)$-free. 

Vinogradov provided an indicator function for when $\gamma \in \F_{q^r}$ is $e$-free, explored in more detail in~\cite[Lemma 2]{mcgown2016grosswald}. Let $\chi_d$ denote an arbitrary character of $\F_{q^r}^*$ of fixed order $d$. Then
\begin{equation} \label{Vinogradov}
\rho(e)\sum_{d\mid e} \frac{\mu(d)}{\varphi(d)}\sum_{\chi_d} \chi_d(\gamma) = \left\{ \begin{array}{ccc}
1 & \vline & \gamma \ e\text{-free} \\
0 & \vline & \text{otherwise}
\end{array}\right.,
\end{equation} 
where $\rho$ denotes the arithmetic function $\rho:\N \to \C: m \mapsto \varphi(m)/m$. Let the quantity $N(e,\cal{A})$ denote the number of $e$-free elements in some subset $\cal{A} \subseteq \F_{q^r}$ and define the sum $S(\cal{A},\chi):=\sum_{\gamma \in \cal{A}}\chi(\gamma)$. Suppose that, for any divisor $d \mid q^r-1$, we can construct a lower bound
\begin{equation} \label{Katz general}
\vert S(\cal{A},\chi_d) \vert \le K(q,r),
\end{equation}
where $K(q,r)$ is independent of order $d$ and character $\chi_d$. The primary goal of this paper is to provide a general criterion for the existence of a primitive element $\gamma \in \cal{A}$, depending only on the bound $K(q,r)$. In particular, we prove the following. 
\newline
\begin{theorem} (\textbf{The modified prime sieve}) \label{msieve}
Consider a subset $\cal{A}\subseteq \F_{q^r}$ and assume $\text{rad}\left(q^r-1\right) = k \prod_{i=1}^{s_1} p_i\prod_{j=1}^{s_2} l_j$ for some integer $k$ and primes $p_i,l_j$. Define the quantities $ \delta := \delta(p_1,...,p_{s_1}) = 1-\sum_{i=1}^{s_1}\frac{1}{p_i}$ and $ \epsilon := \epsilon(l_1,...,l_{s_2}) = \sum_{j=1}^{s_2}\frac{1}{l_i}$.  Suppose $K(q,r)$ satisfies \eqref{Katz general} and $\delta \rho(k) > \epsilon$. If
\begin{equation} \label{m:sieve condition}
\vert \cal{A} \vert > \frac{\rho(k)W(k)(s_1+2\delta-1) + s_2 - \delta \rho(k) - \epsilon}{\delta \rho(k) - \epsilon}K(q,r), 
\end{equation}
then $\cal{A}$ contains a primitive element.
\end{theorem}
Here and throughout, $W(k)=2^{\omega(k)}$ where $\omega(k)$ denotes the arithmetic function counting unique prime divisors of $k$. If $q$ is a fixed prime power, we adopt the notation $\omega_r:=\omega(q^r-1)$. Variants of the prime sieve where $s_2=0,1$ are abundant in the literature for a variety of subsets $\cal{A} \subset \F_{q^r}$, see for example Cohen~\cite[Proposition 4.3]{cohen2010primitive}, Bailey, Cohen, Sutherland and Trudgian~\cite[Lemma 1]{bailey2019primitive} or Booker, Cohen, Leong and Trudgian~\cite{booker2022primitive}. The first author used the sieve combined with a hybrid lower bound on $\omega(q^r-1)$ to find primitive elements on lines in $\F_{q^r}$, see~\cite{bagger2024hybrid} for details. We demonstrate the utility of the modified prime sieve by tackling a problem of Fernandes and Reis~\cite{FernandesReis2021} with $\cal{A}$ avoiding affine hyperplanes.

Let $C = \{C_1,...,C_r\}$ be a set of $\F_{q} $-affine hyperplanes of $\F_{q^r}$ in general position and consider $\cal{G}_\cal{A}= \F_{q^r} \setminus \bigcup_{i=1}^r C_i$. We show that $\cal{G}_\cal{A}$ contains a primitive element for any choice of affine hyperplanes, for all but the following list of potential exceptional pairs $(q,r)$. Combining our work with \cite[Theorem 1.1]{FernandesReis2021}, we have:
\begin{theorem}\label{thm:mainHyperplaneResult}
    For any $(q,r)$, $\cal{G}_\cal{A}$ contains a primitive element, except possibly in the following cases:
    \begin{itemize}
        \item $q=13$ and $r=4$,
        \item $q=11$ and $r\in \{4,6,12\}$,
        \item $q=9$ and $r\in \{6,8,9,10,12,15\}$,
        \item $q=8$ and $r\in \{6,8,10,12,16,20\}$,
        \item $q=7$ and $r\in \{3,4,\ldots,12\}\cup \{14,15,16,18,20\}$,
        \item $q=5$ and $r\in \{2,3,\ldots, 25\}\cup \{27,30,35\}$,
        \item $q=4$ and $r\in \{3,4,\ldots, 33\}\cup \{35,36,37,39,41,42,43\}\cup\{45,47,49,\ldots,87\}\cup\{91,93,95,99,105,111,115,117,123,135\}$,
        \item \begin{sloppypar}$q=3$ and, $r$ is odd or $r\in \{2,4,6,\ldots, 84\}\cup\{88,90,92,\ldots,104\}\cup\{108,112,114,120,144\}$.\end{sloppypar}
    \end{itemize}
\end{theorem}
For most of these cases, we make no claim about whether these possible exceptions are genuine (i.e. $\cal{G}_\cal{A}$ actually does not contain a primitive element for some choice of hyperplanes). However, the authors have verified that the cases $(q,r)=$ $(3,2)$, $(5,2)$ and $(3,3)$ \textbf{are genuine exceptions}. 
\newline \newline
The outline of this paper is as follows. In \textsection\ref{sec:Modified} we prove Theorem~\ref{msieve} in the general setting of $\cal{A} \subset \F_{q^r}$. We introduce the modified sieve and associated character sum bound in the hyperplane setting in \textsection\ref{sec:Hyper}. We give explicit upper bounds for $\omega_r$ in the cases where $\cal{A}$ might fail to have a primitive element. This is done by applying the modified sieve first in the general setting in \textsection\ref{sec:Explic}, then with a specific factorisation of $q^r-1$ in \textsection\ref{sec:ExplicSpec}. Finally, we verify some genuine exceptional pairs $(q,r)$ in \textsection\ref{sec:Exceptions}.

\section{The modified prime sieve} \label{sec:Modified}
The goal of this section is to prove Theorem \ref{msieve} by way of Lemmas \ref{VinoN}-\ref{Nepsilonbound}.
\begin{lemma} \label{VinoN}
Let $\chi_d$ denote any character of $\F_{q^r}$ of order $d$. Then
$$ N(e,\cal{A}) =\rho(e)\left(\vert \cal{A} \vert+\sum_{1<d\mid e}\frac{\mu(d)}{\varphi(d)}\sum_{\chi_d}S(\cal{A},\chi_d)\right).$$
\end{lemma}
\begin{proof}
This follows from considering the indicator function in~\eqref{Vinogradov} for $e$-free elements and summing over $\gamma \in \cal{A}$. Since the sums are finite, we may swap the order of summation. Treating the case when $(d = 1) \mid e$ separately, we obtain the result.
\end{proof}
\begin{lemma}
Suppose $e \mid q^r-1$ and $K(q,r)$ satisfies \eqref{Katz general}. Then
\begin{equation} \label{Cor pt1}
N(e,\cal{A}) \ge \rho(e) \left(\vert \cal{A} \vert - (W(e)-1)K(q,r)\right).
\end{equation} 
If, in addition, $d\mid e$ and $\gcd(d,e/d) = 1$, then
\begin{equation} \label{Cor pt2}
\left\vert N(e,\cal{A}) - \rho\left(\frac{e}{d} \right) N(d,\cal{A}) \right\vert \le \rho(e)\left(W(e)-W(d)\right)K(q,r).
\end{equation}
\end{lemma}
\begin{proof}
For \eqref{Cor pt1}, write $N(e,\cal{A})$ in the form from Lemma~\ref{VinoN} and bound $N(e,\cal{A})$ from below via taking absolute values:
$$ N(e,\cal{A}) \geq \rho(e) \left(\vert \cal{A} \vert - \left\vert \sum_{1<d\mid e}\frac{\mu(d)}{\varphi(d)}\sum_{\chi_d}S(\cal{A},\chi_d) \right\vert\right). $$
The number of squarefree divisors $d$ of $e$ is given by $W(e)$ and the non-vanishing of $\mu(d)$ coincides exactly with these divisors. By repeated application of the triangle inequality and \eqref{Katz general}, the bound \eqref{Cor pt1} follows. For \eqref{Cor pt2},  write $N(e,\cal{A}) - \rho\left(\frac{e}{d}\right) N(d,\cal{A})$ in the form from Lemma~\ref{VinoN}. Since $e/d$ and $d$ are coprime, we have $\rho(e/d)\rho(d) = \rho(e)$. We take absolute values and apply \eqref{Katz general}, completing the proof.
\end{proof}
\begin{lemma} \label{setRes}
Let $e_1,e_2 \mid q^r-1$, $E = \textnormal{lcm} (e_1,e_2)$ and $e = \gcd(e_1,e_2)$. Then
$$ N(E,\cal{A}) \geq N(e_1,\cal{A})+N(e_2,\cal{A})-N(e,\cal{A}).$$
\end{lemma}
\begin{proof}
We claim the following two statements: for any $\gamma\in\cA$,\begin{align*}
    \gamma\text{ is } e_1\text{- and }e_2\text{-free}&\iff \gamma\text{ is }E\text{-free}\\
    \gamma\text{ is not }e\text{-free}&\implies \gamma\text{ is neither }e_1\text{- nor }e_2\text{-free}.
\end{align*}
Assume arbitrary $\gamma \in \cal{A}$ is $e_1$ and $e_2$-free. If $\gamma = \beta^d$ for some $d \mid E$, then write $d = d_1d_2$ with $d_1\mid e_1$ and $d_2 \mid e_2$. Thus $\gamma=(\beta^{d_1})^{d_2}$ so $d_2=1$ since $\gamma$ is $e_2$-free. Similarly, $d_1=1$ and $\gamma$ is $E$-free. If instead $\gamma$ is $E$-free, then $\gamma=\beta^{d}$ for $d\mid e_1$ implies $d\mid E$ so $d=1$. Similarly, $\gamma$ is $e_2$-free. Assume instead that $\gamma$ is not $e$-free. Then $\gamma = \beta^d$ for some $1<d\mid e$. Then $d \mid e_1,e_2$ so $\gamma$ is neither $e_1$ or $e_2$-free. This proves the claim. Thus, if we denote by $\cal{A}^d$ the $d$-free elements in $\cal{A}$, we have proven that 
$$ \cal{A}^{e_1} \cup \cal{A}^{e_2} \subseteq \cal{A}^e \quad \text{and} \quad \cal{A}^{e_1} \cap \cal{A}^{e_2} = \cal{A}^E. $$
By considering cardinalities of the finite sets $\cal{A}^d$, we have 
$$N(e_1,\cal{A})+N(e_2,\cal{A}) = \left\vert \cal{A}^{e_1} \cup \cal{A}^{e_2} \right\vert + \left\vert \cal{A}^{e_1} \cap \cal{A}^{e_2} \right\vert \leq \vert \cal{A}^e\vert + \vert \cal{A}^E\vert = N(e,\cal{A}) + N(E,\cal{A}).$$
\end{proof}
\begin{lemma} \label{Ndeltabound}
Let $e \mid q^r-1$ and suppose $\text{rad}(e) = k \prod_{i=1}^{s} p_i$ for a choice of $k$. Then
\begin{align}
N(e,\cal{A}) & \ge (1-s) N(k, \cal{A}) + \sum_{i=1}^s N(k p_i,\cal{A}) \\
& = \delta(p_1,\dots,p_s)N(k,\cal{A}) + \sum_{i=1}^s\left(N(kp_i,\cal{A}) - \left(1-\frac{1}{p_i}\right)N(k,\cal{A}) \right).
\end{align} 
\end{lemma}
\begin{proof}
Observe that $N(e,\cal{A}) = N(\text{rad}(e),\cal{A})$. Indeed, we have
$$ \cal{A}^{\text{rad}(e)} \supseteq \cal{A}^{e} \cap \cal{A}^{\text{rad}(e)} = \cal{A}^{\text{lcm}(e,\text{rad}(e))} = \cal{A}^e. $$
Assume $\gamma$ is $\text{rad}(e)$-free and $\gamma = \beta^d$ for some $d \mid e$. Letting $e = \prod_{i=1}^sp_i^{a_i}$ and $d = \prod_{i=1}^{s'} p_i^{b_i}$ for $1\le b_i\le a_i$ and $s' \le s$, we write $\gamma = \beta^d = (\beta')^{\prod_{i=1}^{s'} p_i}$, say. Since $\gamma$ is $\text{rad}(e)$-free and $\prod_{i=1}^{s'} p_i \mid \text{rad}(e)$, we conclude $\gamma = \beta' = \beta^{\prod_{i=1}^{s'}p_i^{b_i-1}}$. We can repeat the argument and reduce the multiplicity of $p_i$ in $d$ further, until concluding that $d=1$. This shows that $\gamma$ is $e$-free and indeed $\cal{A}^e = \cal{A}^{\text{rad}(e)}$, as claimed. We proceed by applying Lemma~\ref{setRes} repeatedly:
\begin{align*}
N(\text{rad}(e),\cal{A}) \ge &  N(kp_1,\cal{A}) + \biggr(N(kp_2\cdots p_s,\cal{A}) - N(k,\cal{A})\biggr) \\
\ge &  N(kp_1,\cal{A}) + N(kp_2,\cal{A}) + \biggr(N(kp_3\cdots p_s,\cal{A}) - 2\cdot N(k,\cal{A}) \biggr) \\
& \vdots \\
\ge &  \sum_{i=1}^s N(k p_i,\cal{A}) -(s-1)N(k,\cal{A}).
\end{align*} 
The second equality is a direct rearrangement of the first. 
\end{proof}
\begin{lemma} \label{Nepsilonbound}
Let $e \mid q^r-1$ and suppose $\text{rad}(e) = k \prod_{i=1}^{s_1} p_i\prod_{j=1}^{s_2}l_j$. Then
\begin{align*}
N(e,\cal{A}) & \ge N(kp_1\cdots p_{s_1},\cal{A}) - \epsilon(l_1,\dots,l_{s_2}) \vert \cal{A} \vert + \sum_{j=1}^{s_2}\left(N(l_j,\cal{A}) - \left(1-\frac{1}{l_j}\right)\vert\cal{A}\vert \right).
\end{align*} 
\end{lemma}
\begin{proof}
Let $k_1 = k p_1\dots p_{s_1}$. Applying Lemma~\ref{Ndeltabound}, we have
\begin{align*}
N(e,\cal{A}) \ge &  (1-s_2) N(k_1, \cal{A}) + \sum_{j=1}^{s_2} N(k_1 l_j,\cal{A}) \\
= &  N(kp_1\dots p_{s_1},\cal{A}) + \sum_{j=1}^{s_2} \biggr(N(k_1l_j,\cal{A})-N(k_1,\cal{A})\biggr).
\end{align*} 
Applying Lemma~\ref{setRes} with $e_1 = k_1$ and $e_2 = l_j$ for each $j \in [1,s_2]$, we obtain
\begin{align*}
N(e,\cal{A}) \ge & N(kp_1\dots p_{s_1},\cal{A}) + \sum_{j=1}^{s_2} \biggr(N(k_1l_j,\cal{A})-N(k_1,\cal{A})\biggr) \\
\ge &  N(kp_1\dots p_{s_1},\cal{A}) + \sum_{j=1}^{s_2} \biggr(N(l_j,\cal{A})-N(1,\cal{A})\biggr) \\
\ge & N(kp_1\dots p_{s_1},\cal{A}) - \epsilon(l_1,\dots ,l_{s_2}) \vert\cal{A}\vert + \sum_{j=1}^{s_2}\frac{1}{l_j}\vert\cal{A}\vert + \sum_{j=1}^{s_2}\biggr(N(l_j,\cal{A})-\vert\cal{A}\vert\biggr), 
\end{align*} 
which completes the proof.
\end{proof}
\begin{proof}[Proof of Theorem 1]
Applying \eqref{Cor pt2} with $e=kp_i$ and $d=k$, and noting $\rho(p_i) = 1-p_i^{-1}$ we obtain
\begin{align*}
\left\vert N(kp_i,\cal{A}) - \left(1-\frac{1}{p_i}\right)N(k,\cal{A}) \right\vert \le & \rho(kp_i)\left(W(kp_i)-W(k)\right)K(q,r) \\
= & \rho(k)\rho(p_i)\left(W(k)W(p_i)-W(k)\right)K(q,r) \\
= & \rho(k)W(k)K(q,r)\rho(p_i).
\end{align*}
Summing over all $i\in [1,s_1]$, we obtain
\begin{align*}
\sum_{i=1}^{s_1}\left\vert N(kp_i,\cal{A}) - \left(1-\frac{1}{p_i}\right)N(k,\cal{A}) \right\vert \le & \sum_{i=1}^{s_1} \rho(k)W(k)K(q,r)\rho(p_i) \\
\le & \rho(k)W(k)K(q,r)\sum_{i=1}^{s_1} \left(1-\frac{1}{p_i}\right) \\
= & \rho(k)W(k)K(q,r)\left(s_1 + \delta - 1\right). 
\end{align*}
By the same logic, applying \eqref{Cor pt2} with $e=l_j$ and $d=1$, before subsequently summing over all $j\in [1,s_2]$, we obtain
\begin{align*}
\sum_{j=1}^{s_2}\left\vert N(l_j,\cal{A}) - \left(1-\frac{1}{l_j}\right)\vert\cal{A}\vert \right\vert = & \sum_{j=1}^{s_2}\left\vert N(l_j,\cal{A}) - \left(1-\frac{1}{l_j}\right)N(1,\cal{A}) \right\vert \\
\le & \sum_{j=1}^{s_2}\rho(l_j)\left(W(l_j)-W(1)\right) K(q,r) \\
= & (s_2-\epsilon) K(q,r). 
\end{align*}
Applying \eqref{Cor pt1} with $e=k$, we obtain
$$ N(k,\cal{A}) \ge \rho(k) \left(\vert \cal{A} \vert - (W(k)-1)K(q,r)\right). $$
By Lemma~\ref{Nepsilonbound} with $e=q^r-1$, we have
\begin{align*}
N(q^r-1,\cal{A}) & \ge N(kp_1\cdots p_{s_1},\cal{A}) - \epsilon \vert \cal{A} \vert + \sum_{j=1}^{s_2}\left(N(l_j,\cal{A}) - \left(1-\frac{1}{l_j}\right)\vert\cal{A}\vert \right).
\end{align*}
Applying Lemma~\ref{Ndeltabound} to $ N(kp_1\cdots p_{s_1},\cal{A})$, we arrive at
\begin{align*}
N(q^r-1,\cal{A}) & \ge \delta N(k,\cal{A}) + \sum_{i=1}^{s_1}\left(N(kp_i,\cal{A}) - \left(1-\frac{1}{p_i}\right)N(k,\cal{A}) \right) \\
& - \epsilon \vert \cal{A} \vert + \sum_{j=1}^{s_2}\left(N(l_j,\cal{A}) - \left(1-\frac{1}{l_j}\right)\vert\cal{A}\vert \right).
\end{align*}
We use \eqref{Cor pt1} when $e=k$ and apply our bounds for the absolute value of the sums:
\begin{align*}
N(q^r-1,\cal{A}) & \ge \delta \rho(k) \left(\vert \cal{A} \vert - (W(k)-1)K(q,r)\right) - \epsilon \vert \cal{A} \vert \\
& -\rho(k)W(k)K(q,r)\left(s_1 + \delta - 1\right) - (s_2-\epsilon) K(q,r) \\
& = (\delta \rho(k) - \epsilon) \vert A \vert - \left(\rho(k) W(k)(s_1+2\delta-1) + s_2 - \delta \rho(k) - \epsilon \right) K(q,r).
\end{align*}
Thus, to ensure $N(q^r-1,\cal{A})>0$, it is sufficient to satisfy
$$ (\delta \rho(k) - \epsilon) \vert \cal{A} \vert > \left(\rho(k) W(k)(s_1+2\delta-1) + s_2 - \delta \rho(k) - \epsilon \right) K(q,r). $$
Assuming $\delta \rho(k)> \epsilon$, and dividing through, the result follows.
\end{proof}
For computations, it is often useful to have special cases of Theorem~\ref{msieve}. 
\begin{cor} (\textbf{The prime sieve}) \newline \label{psieve}
Let $q$ be a prime power, $r \in \N$ and $\cal{A}\subseteq \F_{q^r}$. Choose a core $k \in \N$ and $p_i \in \P$ such that $\text{rad}\left(q^r-1\right) = k \prod_{i=1}^s p_i$. Suppose $K(q,r)$ satisfies \eqref{Katz general} and $\delta > 0$. If 
$$ \vert \cal{A} \vert > \frac{W(k)\left(s+2\delta -1 \right)-\delta}{\delta}K(q,r), $$
then $\cal{A}$ contains a primitive element.
\end{cor}
\begin{proof}
Apply Theorem~\ref{msieve} for $s_1 = s$ and $s_2 = 0$. 
\end{proof}
\begin{cor} (\textbf{The unsieved bound}) \newline \label{unsieved}
Let $q$ be a prime power, $r \in \N$ and $\cal{A}\subseteq \F_{q^r}$. Suppose $K(q,r)$ satisfies \eqref{Katz general}. If
$$ \vert \cal{A} \vert > \left(W(q^r-1)-1\right)K(q,r), $$
then $\cal{A}$ contains a primitive element.
\end{cor}
\begin{proof}
Apply Theorem~\ref{msieve} for $s_1 = s_2 = 0$. 
\end{proof}
In general, when applying Theorem~\ref{msieve} computationally, the goal is to choose a core $k$ and primes $p_i,l_j$ such that the right-hand side is minimised. For any fixed $s_1,s_2$, we take $l_j$ to be the $s_2$ largest factors of $\text{rad}(q^r-1)$ and $p_i$ to be the $s_1$ largest factors of $\text{rad}(q^r-1)/\prod_{j=1}^{s_2}l_j$. We then iterate through the pairs $(s_1,s_2)$ for which $\delta\rho(k)>\epsilon$, using the lowest bound possible. If the factorisation of $q^r-1$ is known, we apply the result with explicit choices of $p_i$ and $l_j$. If instead the factorisation of $q^r-1$ is unknown, we can still improve over the unsieved bound by taking $l_{s_2-j}$ and $p_{s_1-i}$ to be the $(\omega_r-j)^\text{th}$ and $(\omega_r-s_2-i)^\text{th}$ primes, respectively.

\section{Avoiding hyperplanes} \label{sec:Hyper}
We now turn to a problem of finding primitive elements outside affine hyperplanes. Fernandes and Reis~\cite{FernandesReis2021} proved the following:
\begin{prop}\cite[Th. 1.1]{FernandesReis2021}
Let $r\ge 2$. Then $\cal{G}_\cal{A}$ contains a primitive element provided one of the following holds:
\begin{flalign*}
\textbf{i) } & q\ge 16 \\
\textbf{ii) } & q=13, \text{ and } r\neq 4 \\
\textbf{iii) } & q=11 \text{ and } r\neq 4,6,12 \\
\textbf{iv) } & q = 7,8,9 \text{ and } r \text{ sufficiently large}.
& \end{flalign*}
\end{prop}
 Recently, Grzywaczyk and Winterhof~\cite[Cor. 3.5]{GW2024} proved that $\cal{G}_\cal{A}$ also contains a primitive element when $q=3,4,5$ and $r$ sufficiently large. Their argument for $q=3$ is derived from recent work by Iyer and Shparlinski~\cite{iyer2024character}. In the case when $r$ is even, Cheng and Winterhof~\cite{cheng2025character} gave explicit bounds on $r$ for $q=3,4,5,7,8,9$. Existing methods essentially rely on proving a character sum estimate $K(q,r)$ satisfying \eqref{Katz general} and applying the unsieved criterion in Corollary~\ref{unsieved}. We use the same estimates, but instead apply the prime sieve in Corollary~\ref{psieve} to obtain substantial improvements.
\begin{lemma} \label{Katz}
For any prime power $q$ and any integer $r\geq 2$,
$$ \left\vert S(\cal{G}_A,\chi) \right\vert \leq \sqrt{3}(q-1)^{r/2}q^{\lceil3r/4\rceil/2}.$$
\end{lemma}
\begin{proof}
From \cite[Th. 2.2]{GW2024}, we have $$ \left\vert S(\cal{G}_A,\chi) \right\vert \leq (q-1)^{r/2}(2q^{3r/2-k}+q^k)^{1/2}$$ for any integer $k$ between 1 and $r$. If we choose $k=\lceil 3r/4 \rceil$, then we get the result.
\end{proof}
\begin{rem}\label{remarkRemoveCeiling}
Actually, $\lceil 3r/4 \rceil$ is not necessarily the best choice for $k$ (noting that we want the right hand side to be as small as possible). Instead, choose $k'\in \{1,\ldots,r\}$ such that
$$2q^{3r/2-k'}+q^{k'} \le 2q^{3r/2-k}+q^k \quad \text{for any } k\in\{1,\ldots,r\}, $$ 
and one can show by differentiating that this is either the floor or ceiling of $\frac{3r}{4}+\frac{\log{2}}{2\log{q}}$. When we know $q$ and $r$, we can say exactly what the best choice of $k$ is, and we will use it instead of $\lceil 3r/4 \rceil$.
\end{rem}
\begin{lemma}\label{Katz2} For any prime power $q$ and even integer $r\geq 2$,
$$\left\vert S(\cal{G}_A,\chi) \right\vert < 2(q-1)^{3r/4}q^{r/8}.$$
\end{lemma}
\begin{proof}
See \cite[Eq. 3.1]{cheng2025character}.
\end{proof}
As a direct result of Corollary~\ref{psieve}, we obtain a sieving criterion for primitive elements avoiding hyperplanes.
\begin{prop} (\textbf{The hyperplane sieves}) \newline \label{hypersieve}
Let $q$ be a prime power and $r$ an integer not less than $2$. Choose a core $k \in \N$ and $p_i \in \P$ such that $\textup{rad}\left(q^r-1\right) = k \prod_{i=1}^s p_i$. Suppose $\delta:=\delta(p_1,\dots,p_s) >0$ and 
\begin{equation}\label{eq:hypersieveInequality}
\frac{(q-1)^{r}}{q^{\lceil 3r/4 \rceil}}  >  3\left[2^{\omega_r -s}\left(\frac{s-1}{\delta }+2\right)-1\right]^2,
\end{equation}
then $\cal{G}_\cal{A}$ contains a primitive element. Further, if $r$ is even and \begin{equation} \label{eq:hypersieveInequalityEven}
\frac{(q-1)^{r/2}}{q^{r/4}}  > 4\left[2^{\omega_r -s}\left(\frac{s-1}{\delta }+2\right)-1\right]^2, 
\end{equation}
then $\cal{G}_A$ contains a primitive element.
\end{prop}
\begin{proof}
Apply Corollary~\ref{psieve} with $\cal{A} = \cal{G}_A$ and $K(q,r) = \sqrt{3}(q-1)^{r/2}q^{\lceil3r/4\rceil/2}$, which satisfies \eqref{Katz general} by Lemma~\ref{Katz}. Squaring both sides and rearranging terms, we obtain \eqref{eq:hypersieveInequality}. We derive \eqref{eq:hypersieveInequalityEven} in the same way using Lemma~\ref{Katz2} in place of Lemma~\ref{Katz}.
\end{proof}

\section{Explicit bounds on \texorpdfstring{$\omega_r$}{ω(r)}} \label{sec:Explic}
For each $q$, we determine an integer $B$, such that $N(q^r-1,\cal{G}_\cal{A})>0$ whenever $\omega_r\geq B$. Such a bound enables us to verify~\eqref{eq:hypersieveInequality} and \eqref{eq:hypersieveInequalityEven} computationally. It is straightforward to find such an $B$, but when $q=3,4$, the naive approach yields unfeasibly large bounds. In those two cases we develop a leaping technique to reduce $B$ to a computable range.

\subsection{The case where \texorpdfstring{$5\leq q \leq 9$}{5≤q≤9}}
\begin{prop} \label{unsievedNaivePrimes}
Let $q$ be a prime power and choose $r\in \N$. If $$\frac{1}{q}(p_1 p_2\ldots p_{\omega_r} )^c>3\times4^{\omega_r},$$ where $c=\log_q(q-1)-\frac{3}{4}$, and $p_j$ is the $j^\text{th}$ prime, then $\cal{G}_\cal{A}$ contains a primitive element. Further, if the hypothesis holds for $\omega_r=N$ and we have $p_N^c>4$, then $\cal{G}_\cal{A}$ contains a primitive element whenever $\omega_r\geq N$.
\end{prop}
\begin{proof}
    Apply Corollary \ref{unsieved} with $\cal{A}=\cal{G}_A$ and $K(q,r)=\sqrt{3}(q-1)^{r/2}q^{\lceil 3r/4 \rceil /2}$ to find that $\cal{G}_\cal{A}$ will contain a primitive element whenever 
    $$(q-1)^r>2^{\omega_r}\sqrt{3}(q-1)^{r/2}q^{\lceil 3r/4 \rceil /2}.$$ Notice that we have $q^r-1\geq p_1 p_2 \ldots p_{\omega_r}$. Therefore,
    \begin{align*}
        \cal{G}_\cal{A}\text{ contains a primitive element }&\impliedby(q-1)^{r/2}>2^{\omega_r}\sqrt{3}q^{(3r/4+1)/2} \\
        &\impliedby \frac{(q-1)^{r/2}}{q^{3r/8}q^{1/2}}>2^{\omega_r}\sqrt{3} \\
        &\impliedby \frac{1}{q}\left(\frac{q-1}{q^{3/4}}\right)^r>3\times 4^{\omega_r} \\
        &\impliedby \frac{1}{q}(q^c)^r>3\times 4^{\omega_r} \\
        &\impliedby \frac{1}{q}(p_1 p_2\ldots p_{w_r}+1)^c>3\times 4^{\omega_r},
    \end{align*}
    which is implied by the hypothesis. If $\frac{1}{q}(p_1 p_2\ldots p_{N} )^c>3\times4^{N}$, then 
    \begin{align*}
        \frac{1}{q}(p_1 p_2\ldots p_{N} p_{N+1})^c&=\frac{1}{q}(p_1 p_2 \ldots p_{N})^c(p_{N+1})^c\\
        &>\frac{1}{q}(p_1 p_2 \ldots p_{N})^c(p_{N})^c \\
        &>3\times4^N\times 4=3\times4^{N+1},
    \end{align*} and the result holds by induction, because $p_{N+1}^c$ is also greater than 4.
\end{proof}

It is simple to verify that Proposition~\ref{unsievedNaivePrimes} gives, for each $q\in\{5,7,8,9\}$, the values of $N$ in Table \ref{table:naivePrimesLimit}.

\begin{table}[h!]
\centering
 \begin{tabular}{|c|c|c|c|c|} 
    \hline
         $q$  & 5&  7&  8& 9\\
         \hline
         $N$ & 61367& 1316& 756& 541\\
        \hline
    \end{tabular}
\caption{Given $q$, a value $N$ such that $\cal{G}_\cal{A}$ contains a primitive element whenever $\omega_r\geq N$.}
\label{table:naivePrimesLimit}
\end{table}

\subsection{The case where \texorpdfstring{$q=4$}{q=4}}
\begin{lemma} \label{sillyUpperBoundq4}
    When $q=4$, the subset $\cal{G}_\cal{A}$ contains a primitive element whenever $r\geq 3.07\times10^{19}$. Thus, if $\cal{G}_\cal{A}$ does not contain a primitive element, $\omega_r < 1.31\times10^{18}$.
\end{lemma}
\begin{proof}
    Using Corollary \ref{unsieved} with $K(q,r)=\sqrt{3}(q-1)^{r/2}q^{\lceil 3r/4 \rceil/2}$ and  $\left\vert\cal{G}_\cal{A}\right\vert=(q-1)^r$, we argue as follows.
    \begin{align*}
        \cal{G}_\cal{A}\text{ contains a primitive element}&\impliedby (q-1)^r>W(q^r-1)\sqrt{3}(q-1)^{r/2}q^{\lceil 3r/4 \rceil/2} \\
        &\impliedby W(q^r-1)<\frac{(q-1)^{r/2}}{\sqrt{3}q^{\lceil 3r/4 \rceil/2}}\\
        & \impliedby q^{\frac{1.92r}{\log{\log{q^r}}}}<\frac{(q-1)^r}{3q^{(3r/4)+1}}=\frac{1}{3q}\left(\frac{q-1}{q^{3/4}}\right)^r,
    \end{align*} 
    where we have squared and used \cite[Lemma 2.2]{FernandesReis2021} at the last step. Taking logarithms, we have
    \begin{align*}
        \cal{G}_\cal{A}\text{ contains a primitive element}&\impliedby \frac{1.92r\log{q}}{\log{r}+\log{\log{q}}}<-\log{(3q)}+r\log\frac{q-1}{q^{3/4}} \\
        &\impliedby \frac{1.92\log{q}}{\log{r}+\log{\log{q}}}<-\frac{\log{(3q)}}{r}+\log\frac{q-1}{q^{3/4}}.
    \end{align*}
    With respect to $r$, the left hand side is decreasing and the right hand side is increasing. If the inequality holds for some $r_0$, it holds for any $r\ge r_0$. It is easily verified, then, that the inequality holds for every $r\geq 3.07\times10^{19}$. If $\cal{G}_\cal{A}$ does not contain a primitive element, then $r<3.07\times 10^{19}$. Using \cite[Lemma 2.2]{FernandesReis2021} again, we find that 
    \begin{align*}
        \omega_r=\log_2W(q^r-1)&<\log_2 \left( q^{\frac{0.96r}{\log{r}+\log{\log{q}}}} \right) \\
        &=\frac{0.96r\log_2q}{\log{r}+\log{\log{q}}}<1.31\times 10^{18},
    \end{align*} where we have again used the fact that $\frac{0.96r\log_2q}{\log{r}+\log{\log{q}}}$ is increasing with respect to $r$.
\end{proof}
\begin{rem}\label{rem:checkOmegaExplanation}
    This bound on $\omega_r$ is too large for us to apply~\eqref{eq:hypersieveInequality} to, so we must reduce it further. Our strategy in the next proposition (which is also useful if $q>4$) is to fix $\omega \in \N$ and eliminate any $r$ with $\omega_r=\omega$. In particular, if we also fix $0\leq t\leq \omega$, we determine an upper bound for $r$, such that \eqref{eq:hypersieveInequality} fails. By \cite[Lemma 2.2]{FernandesReis2021}, this gives an upper bound on $\omega_r$. If this forces $\omega_r<\omega$, we conclude $N(q^r-1,\cal{G}_\cal{A})>0$ for all such $r$.
\end{rem}
\begin{prop} \label{checkOmega}
    Let $q\geq 4$ be a prime power and $r\in \N$. Fix $\omega$ and $0\leq t\leq\omega$, and let $s(\omega)=\omega-t$. Let $$\delta(\omega)=1-\sum_{i=\omega-s+1}^{\omega}\frac{1}{p_i}=1-\sum_{i=t+1}^{\omega}\frac{1}{p_i},$$ with $p_i$ the $i^\text{th}$ prime, and assume $\delta>0$. Let 
    \begin{align*}
    R(\omega)&=\log{\left(3\left[2^t\left(\frac{s(\omega)-1}{\delta(\omega)}+2\right)-1\right]^2\right)}, \\
    K(R)&=\frac{R+\log{q}}{\log{(q-1)}-\frac{3}{4}\log{q}}, \\
    B(K)&=\frac{0.96K\log_2{q}}{\log K+\log\log q}.
    \end{align*}
    If $\omega\geq (B\circ K\circ R)(\omega)$, then $\cal{G}_\cal{A}$ contains a primitive element. Additionally, if we fix $t$ but not $\omega$, and suppose $\omega\geq t$, then $(K\circ R)$ and $R$ are strictly increasing functions of $\omega$.
\end{prop}
\begin{proof}
    First, note that $R(\omega)$ is the logarithm of the right hand side in Proposition \ref{hypersieve}, except that we have (possibly) changed the primes in the sum constituting $\delta$. Fix $t$ and $\omega$, and suppose that we will always take the largest $s(\omega)$ primes in the factorisation of $\textup{rad}(q^r-1)$, denoting these by $p'_1,p'_2,\ldots,p'_{s(\omega)}$. We have that 
    $$\delta(\omega)=1-\sum_{i=\omega-s(\omega)+1}^{\omega}\frac{1}{p_i}\leq 1-\sum_{i=1}^{s(\omega)}\frac{1}{p'_i}:=\delta'(\omega),$$ 
    from which it follows that 
    $$\left(3\left[2^t\left(\frac{s(\omega)-1}{\delta(\omega)}+2\right)-1\right]^2\right) \geq \left(3\left[2^t\left(\frac{s(\omega)-1}{\delta'(\omega)}+2\right)-1\right]^2\right). $$ 
    If $\frac{(q-1)^r}{q^{\lceil 3r/4 \rceil}}$ is greater than the left hand side, then it is also greater than the right hand side, and by Proposition \ref{hypersieve} there must then be a primitive element in $\cal{G}_\cal{A}$.
    It is now sufficient to show that $\log(\frac{(q-1)^r}{q^{\lceil 3r/4 \rceil}})>R(\omega)$, which will happen if 
    $$ r\log{(q-1)}-\left(\frac{3r}{4}+1\right)\log{q}>R(\omega),$$ 
    or equivalently,
    \begin{align}
        r>\frac{R(\omega)+\log{q}}{\log{(q-1)}-\frac{3}{4}\log{q}}=(K\circ R)(\omega). \label{winIfrGreaterThan}
    \end{align} Note that this equivalence requires the denominator in (\ref{winIfrGreaterThan}) to be positive, which holds because $q\geq 4$. For a contradiction, suppose there is no primitive element in $\cal{G}_\cal{A}$. Then $r\leq (K\circ R)(\omega)$. By~\cite[Lemma 2.2]{FernandesReis2021}, 
    \begin{align*}
    \omega=\log_2 W(q^r-1)&<\frac{0.96}{\log\log(q^r)}\log_2(q^r) \\
    &=\frac{0.96r\log_2(q)}{\log r+\log\log q}.
    \end{align*}
    This bound is effectively a Robin bound for $\omega(n)$, see~\cite{robin1983tchebychef} for more details. It is easy to show that this function is increasing on $r$ as long as $r>\frac{e}{\log{q}}$ and $q>1$. Since $r\geq 2$ and $q\geq 4$, these inequalities hold. It follows that $$\omega<\frac{0.96r\log_2(q)}{\log{r}+\log{\log{q}}}\leq \frac{0.96 (K\circ R)(\omega) \log_2(q)}{\log{(K\circ R)(\omega)}+\log{\log{q}}}=(B\circ K \circ R)(\omega).$$ By assumption, this is a contradiction. This completes the proof that $\cal{G}_\cal{A}$ contains a primitive element. Finally, it is clear that if $R$ is strictly increasing with respect to $\omega$, then so is $K\circ R$. Clearly $s(\omega)=\omega-t$ strictly increases with $\omega$, so we just need to show that $\delta$ is strictly decreasing with respect to $\omega$. But it is clear that $$1-\sum_{i=t+1}^{\omega+1}\frac{1}{p_i}=1-\frac{1}{p_{\omega+1}}-\sum_{i=t+1}^{\omega} \frac{1}{p_i}< 1-\sum_{i=t+1}^{\omega}\frac{1}{p_i},$$ so we are done.
\end{proof}
\begin{rem}\label{rem:leapingExplanation}
    We describe our leaping technique in the next lemma. Essentially, the strategy is to assume a lower bound on $\delta$, then show that if the criterion in Proposition \ref{checkOmega} holds for $\omega$, then it holds for $\omega+1$. By induction, the criterion holds for as long as the lower bound on $\delta$ does. This allows us to apply Proposition \ref{checkOmega} for a large number of values of $\omega$, all at once.
\end{rem}
\begin{lemma} \label{checkManyOmegasq4}
Let $q=4$, select some $\omega_0$ and fix an integer $t$ between $0$ and $\omega_0$. Set $\delta(\omega)=1-\sum_{i=t+1}^{\omega}(1/p_i)$ and suppose that we have $\delta(\omega_0)>m$ for some $0<m$. Let $s(\omega)=\omega-t$, and
\begin{align*}
    r(\omega)&=\log{\left(3\left[2^t\left(\frac{s(\omega)-1}{m}+2\right)\right]^2\right)}, \\
    K(R)&=\frac{R+\log{q}}{\log{(q-1)}-\frac{3}{4}\log{q}}, \\
    B(K)&=\frac{0.96K\log_2{q}}{\log K+\log\log q},
\end{align*} where we have replaced the $\delta$ in Proposition \ref{checkOmega} by its lower bound $m$, and dropped the $-1$. Suppose that $\omega_0\geq (B\circ K \circ r)(\omega_0)$, $s(\omega_0)\geq 19$, $(K\circ r)(\omega_0)\geq 6$ and $(K\circ R)(\omega_0)\geq 2$. If we find that $\delta(\omega_0+1)>m$, then there is a primitive element in $\cal{G}_\cal{A}$ for $\omega=\omega_0+1$. Further, if $\delta(\omega_1)>m$ for some $\omega_1> \omega_0$ then there is a primitive element in $\cal{G}_\cal{A}$ whenever $\omega$ is between $\omega_0+1$ and $\omega_1$.
\end{lemma}

\begin{proof}
    We know that $\omega_0\geq (B\circ K\circ r)(\omega_0)$, and we will now estimate $(B\circ K\circ r)(\omega_0+1)$. First, we examine $r(\omega_0+1)$.
    We have
    \begin{align*}
        r(\omega_0+1)-r(\omega_0)&=2\left(\log{2^t}+\log{\left(\frac{s(\omega_0+1)-1}{m}+2\right)}\right)\\ &- 2\left(\log{2^t}+\log{\left(\frac{s(\omega_0)-1}{m}+2\right)}\right)\\
        &= 2\left(\log{\left(\frac{s(\omega_0+1)-1}{m}+2\right)}-\log{\left(\frac{s(\omega_0)-1}{m}+2\right)}\right) \\
        &= 2\left(\log{\left(\frac{\omega_0-t}{m}+2\right)}-\log{\left(\frac{\omega_0-t-1}{m}+2\right)}\right).
    \end{align*}
    The argument of the first log exceeds the argument of the second by $\frac{1}{m}$. log is increasing, its second derivative is always negative and its gradient at $x$ is $\frac{1}{x}$, so the difference in the last line cannot exceed $$\frac{1}{m}\times\frac{1}{\frac{\omega_0-t-1}{m}+2}.$$
    We then have \begin{align*}
        r(\omega_0+1)-r(\omega_0)&\leq \frac{2}{m}\frac{1}{\frac{\omega_0-t-1}{m}+2}=\frac{2}{\omega_0-t-1+2m}\\
        &\leq\frac{2}{s(\omega_0)-1}\leq\frac{1}{9}.
    \end{align*} It should also be clear that $r(\omega_0+1)-r(\omega_0)>0$.
    It follows immediately that \begin{align*}
        0<(K\circ r)(\omega_0+1)-(K\circ r)(\omega_0)\leq \frac{1/9}{\log{(q-1)}-\frac{3}{4}\log{q}}< 2
    \end{align*} by substituting $q=4$. 
    We now examine the properties of the function $B$. We have \begin{align*}
    B(K)&=\frac{1.92K}{\log{K}+\log{\log{4}}}\text{ since}\log_2{4}=2, \\
    B'(K)&=\frac{1.92}{\log{K}+\log{\log{4}}}\left(1-\frac{1}{\log{K}+\log{\log{4}}}\right)\text{, and} \\
    B''(K)&=\frac{1.92}{K(\log{K}+\log{\log{4}})^2}\left(\frac{2}{\log{K}+\log{\log{4}}}-1\right).
    \end{align*}
    It follows that $B(K)$ is strictly increasing if $\log{K}+\log{\log{4}}>1$, i.e. $K>\frac{e}{\ln{4}}\approx 1.961$. $B(K)$ is concave down if $\log{K}+\log{\log{4}}>2$, i.e. $K>\frac{e^2}{\ln{4}}\approx 5.330$. When $K\geq 6$, we have $0<B'(K) \leq B'(6)$ by concavity and the fact that $B$ is increasing. By evaluating $B'(6)$ we have $0<B'(K)<\frac{1}{2}$, valid for $K\geq 6$.

    Thus, when moving from $\omega_0$ to $\omega_0+1$, $K\circ r$ increases by less than 2, so $(B\circ K \circ r)$ increases by less than 1 (as the gradient of B is less than $\frac{1}{2}$ if $(K\circ r)(\omega_0)$ is at least 6, which it is by assumption). That is, $0<(B\circ K\circ r)(\omega_0+1)-(B\circ K\circ r)(\omega_0)<1$. Using this, we have     \begin{equation}\label{eq:caveat1}\omega_0+1\geq (B\circ K\circ r)(\omega_0)+1>(B\circ K \circ r)(\omega_0+1).\end{equation} Notice also that \begin{align}(K\circ r)(\omega_0+1)&>(K\circ r)(\omega_0)\geq 6\label{eq:caveat2},\\
    s(\omega_0+1)&=\omega_0-t+1>\omega_0-t=s(\omega_0)\geq 19\label{eq:caveat3},\text{ and}\\
    (K\circ R)(\omega_0+1)&>(K\circ R)(\omega_0)\geq 2,\label{eq:caveat4}\end{align} by Proposition \ref{checkOmega} and the assumed lower bound on $(K\circ r)(\omega_0)$. 
    
    Next, we have to show that $(B\circ K\circ r)(\omega_0+1)>(B\circ K\circ R)(\omega_0+1)$. Since $\delta(\omega_0+1)>m$, we immediately get $r(\omega_0+1)>R(\omega_0+1)$, which implies 
    $$(K\circ r)(\omega_0+1)>(K\circ R)(\omega_0+1)>(K\circ R)(\omega_0)\geq 2.$$ 
    Since $B$ is strictly increasing with respect to $K$ when $K\geq 2$, we have 
    $$(B\circ K \circ r)(\omega_0+1)>(B\circ K\circ R)(\omega_0+1).$$
    It follows that $\omega_0+1>(B\circ K\circ R)(\omega_0+1)$, so there is a primitive element in $\cG_{\cA}$ for $\omega=\omega_0+1$. The final claim follows by induction, taking note of Equations \eqref{eq:caveat1}, \eqref{eq:caveat2}, \eqref{eq:caveat3}, and \eqref{eq:caveat4}.    
\end{proof}
\begin{prop}~\cite[Th. 2]{Mertens1874} \label{Mertens2} For any $n\geq 2$, we have
    $$\left\vert\sum_{p\leq n}\frac{1}{p}-\log{\log{n}}-M\right\vert\leq \frac{4}{\log{(n+1)}}+\frac{2}{n\log{n}}$$ where $0.26149<M<0.26150$ is the Meissel-Mertens constant.
\end{prop}
\begin{prop}\cite[Th. 1]{RosserSchoenfeld1962} \label{RossTh1}
    For any $x>1$, 
    \begin{equation}
    \pi(x)<\frac{x}{\log{x}}\left(1+\frac{3}{2\log{x}}\right).
    \end{equation} 
    For any $x\geq 59$, 
\begin{equation} \label{eq:Rosser2}
\pi(x)>\frac{x}{\log{x}}\left(1+\frac{1}{2\log{x}}\right).
\end{equation}
\end{prop}
Both Propositions~\ref{Mertens2} and \ref{RossTh1} have since been improved, but the bounds stated here are sharp enough for our applications. 
\begin{prop} \label{goodUpperBoundq4}
    Let $q=4$. If $\cal{G}_\cal{A}$ does not contain a primitive element, then $\omega_r \leq 10^5$.
\end{prop}
\begin{proof}
The proof will follow from applying Lemma \ref{checkManyOmegasq4}. We will show that, when $q=4$, $\cal{G}_\cal{A}$ contains a primitive element whenever $\omega_r>10^5$. Let $\omega_0=10^5$, $m=0.05$, $t=3\times10^4$. Then $\delta(\omega_0)\approx0.9024>0.05=m$. We find that 
$$ r(\omega_0)\approx 41618.2, \ (K\circ r)(\omega_0)\approx 706716.7\geq 6 \text{ and } (B\circ K\circ r)(\omega_0)\approx 98361.3\leq \omega_0.$$ 
Additionally, $s(\omega_0)=10^5-3\times10^4=7\times10^4\geq 19$. The hypotheses in Lemma \ref{checkManyOmegasq4} are satisfied since $(K\circ R)(\omega_0)\approx 706694.0\geq 2$. If we have $\delta(\omega_1)>m$, for some $\omega_1> 10^5$, then there is a primitive element in $\cal{G}_\cal{A}$ for any $\omega$ between $10^5+1$ and $\omega_1$. We now want to know how large $\omega_1$ can be before this hypothesis fails. We require that $1-\sum_{i=3\times10^4}^{\omega_1}(1/p_i)>0.05$, which happens if $$\sum_{i=1}^{\omega_1}\frac{1}{p_i}<0.95+\sum_{i=1}^{3\times 10^4}\frac{1}{p_i}.$$ The right hand side is more than 3.758. Using Proposition \ref{Mertens2} with $n$ equal to the $\omega_1^\text{th}$ prime, we get
\begin{align}
\sum_{i=1}^{\omega_1}\frac{1}{p_i}=\sum_{p\leq n}\frac{1}{p}\leq \log{\log{n}}+M+\frac{4}{\log{(n+1)}}+\frac{2}{n\log{n}},\label{eq:MertensCheck}
\end{align} 
and it is easy to verify that when $n=2.93\times10^{12}$, the right hand side is less than 3.758. Using \eqref{eq:Rosser2}, $\pi(2.93\times10^{12})>1.038\times 10^{11}$, which is to say that the $(1.038\times 10^{11})^\text{th}$ prime is less than $2.93\times 10^{12}$. Taking $w_1=1.038\times 10^{11}$ is therefore valid, and we have shown that $\cal{G}_\cal{A}$ contains a primitive element for any $\omega$ satisfying $10^5+1\leq \omega \leq 1.038\times 10^{11}$.

We now perform another \say{jump}. Let $\omega_0=1.038\times 10^{11}$, $m=0.05$, $t=1.5\times10^7$. $\delta(\omega_0)$ is too large to calculate exactly, so we bound it using Proposition \ref{Mertens2}. We have $$\delta(1.038\times 10^{11})=1-\sum_{i=1.5\times 10^7 +1}^{1.038\times 10^{11}}(1/p_i)=1-\sum_{i=1}^{1.038\times 10^{11}}(1/p_i)+\sum_{i=1}^{1.5\times 10^7}(1/p_i).$$ The first sum does not exceed 3.758, using the bound from the previous step. The second sum is more than 3.228. It follows that $$\delta(1.038\times 10^{11})>1-3.758+3.228=0.47>m,$$ as needed. Also, 
\begin{align*}
r(\omega_0)& \approx  2.08\times 10^7  \\
(K\circ r)(\omega_0) & \approx 3.53\times10^8 \geq 6  \\
(B\circ K\circ r)(\omega_0) & \approx 3.39\times 10^7\leq \omega_0  \\
s(\omega_0) & = 1.038\times 10^{11}-1.5\times 10^7\geq 19. 
& \end{align*}
Finally, since $\delta(\omega_0)<1$, we find that $R(\omega_0)>\log(3[2^t(\frac{s(\omega_0)-1}{1}+2)-1]^2)\approx 2.08\times 10^7$, so $(K\circ R)(\omega_0)>3.53\times 10^8>2$. Now, we have satisfied all of the hypotheses in Lemma \ref{checkManyOmegasq4}, and we seek a large value of $\omega_1$ satisfying $\delta(\omega_1)>m$. As before, we need $$\sum_{i=1}^{\omega_1}\frac{1}{p_i}<0.95+\sum_{i=1}^{1.5\times 10^7}\frac{1}{p_i}.$$ The right hand side is more than $4.1785$. Using Proposition \ref{Mertens2} with $n$ equal to the $\omega_1^\text{th}$ prime, we again get \eqref{eq:MertensCheck}, whose right hand side is less than 4.1785 if $n=1.019\times 10^{20}$. By \eqref{eq:Rosser2}, the $(2.2358\times10^{18})^\text{th}$ prime is less than $1.019\times 10^{20}$, so it is valid to pick $\omega_1=2.2358\times 10^{18}$. Therefore, $\cal{G}_\cal{A}$ contains a primitive element for any $\omega$ satisfying $1.038\times 10^{11}+1 \leq \omega \leq 2.2358\times 10^{18}$. Combining this with the result of the previous step and Lemma \ref{sillyUpperBoundq4}, we obtain the desired result.
\end{proof}
\subsection{The case where \texorpdfstring{$q=3$}{q=3}}
We now turn our attention to the $q=3$ case. We cannot apply \eqref{eq:hypersieveInequality}, because the left hand side does not grow with $r$, owing to the fact that $\frac{(q-1)}{q^{3/4}}<1$. Thus, we apply Criterion \eqref{eq:hypersieveInequalityEven} instead. We will now present results analogous to Lemma \ref{sillyUpperBoundq4}, Proposition \ref{checkOmega} and Lemma \ref{checkManyOmegasq4}, and use a similar \say{jumping} argument to eliminate all but finitely many values of $\omega$. In particular, we construct an explicit upper bound on $\omega_r$ above which we win, in the sense that $\cal{G}_\cal{A}$ must contain a primitive element. For an explanation of the strategy of Proposition \ref{checkOmegaq3} (resp. Lemma \ref{checkManyOmegasq3}), see Remark \ref{rem:checkOmegaExplanation} (resp. Remark \ref{rem:leapingExplanation}).

\begin{lemma} \label{sillyUpperBoundq3}
    Let $q=3$ and suppose that $r$ is even. Then $\cal{G}_\cal{A}$ contains a primitive element whenever $r\geq 4.972\times 10^{12}$. Thus, if $\cal{G}_\cal{A}$ does not contain a primitive element, $\omega_r< 2.58\times 10^{11}$.
\end{lemma}
\begin{proof}
    Use the unsieved bound of Corollary \ref{unsieved} with $K(q,r)=2(q-1)^{3r/4}q^{r/8}$ from Lemma \ref{Katz2}. Then
    \begin{align*}
        N(q^r-1,\cal{G}_\cal{A})>0&\impliedby q^\frac{0.96r}{\log{\log{q^r}}}<\frac{(q-1)^{r/4}}{2q^{r/8}} \\
        &\impliedby \frac{0.96r\log{q}}{\log{r}+\log{\log{q}}}<-\log{2}+\frac{1}{4}r\log{(q-1)}-\frac{1}{8}r\log{q} \\
        &\impliedby \frac{0.96\log{q}}{\log{r}+\log{\log{q}}}<\frac{-\log{2}}{r}+\frac{1}{4}\log{(q-1)}-\frac{1}{8}\log{q}.
    \end{align*} This holds if $r\geq 4.972\times 10^{12}$. If $\cal{G}_\cal{A}$ does not contain a primitive element, then $r<4.972\times 10^{12}$, and using \cite[Lemma 2.2.3]{FernandesReis2021} we find that $\omega_r <2.58\times 10^{11}$.
\end{proof}

\begin{prop} \label{checkOmegaq3}
    Let $q=3$ and let $r$ be an even positive integer that is at least $4$. Fix $\omega$ and $0\leq t\leq\omega$, and let $s(\omega)=\omega-t$. Let $$\delta(\omega)=1-\sum_{i=\omega-s+1}^{\omega}\frac{1}{p_i}=1-\sum_{i=t+1}^{\omega}\frac{1}{p_i},$$ with $p_i$ the $i^\text{th}$ prime, and assume $\delta>0$. Let 
    \begin{align*}
    R(\omega)&=\log{\left(4\left[2^t\left(\frac{s(\omega)-1}{\delta(\omega)}+2\right)-1\right]^2\right)}, \\
    K(R)&=\frac{R}{\frac{1}{2}\log{(q-1)}-\frac{1}{4}\log{q}}, \\
    B(K)&=\frac{0.96K\log_2{q}}{\log K+\log\log q}.
    \end{align*}
    If $\omega\geq (B\circ K\circ R)(\omega)$, then $\cal{G}_\cal{A}$ contains a primitive element. Additionally, if we fix $t$ but not $\omega$, and suppose $\omega\geq t$, then $(K\circ R)$ and $R$ are strictly increasing functions of $\omega$.
\end{prop}
\begin{proof}
    The proof is analogous to that of Proposition \ref{checkOmega}. Note the change from $3$ to $4$ in the definition of $R(\omega)$, and the change in definition of $K(R)$. This occurs because we replace $\frac{(q-1)^r}{3q^{(3r/4)+1}}$ with $\frac{(q-1)^{r/2}}{4q^{r/4}}$. Finally, recall that it was necessary to have $r>\frac{e}{\log{q}}$, which is the reason we exclude the case $r=2$.
\end{proof}

\begin{lemma} \label{checkManyOmegasq3}
Let $q=3$, select some $\omega_0$, suppose $r$ is an even positive integer not less than $4$, and fix a $t$ between $0$ and $\omega_0$. Let $\delta(\omega):=1-\sum_{i=t+1}^{\omega}(1/p_i)$ and suppose that we have $\delta(\omega_0)>m$ for some $0<m$. Let $s(\omega)=\omega-t$, and
\begin{align*}
    r(\omega)&=\log{\left(4\left[2^t\left(\frac{s(\omega)-1}{m}+2\right)\right]^2\right)}, \\
    K(R)&=\frac{R}{\frac{1}{2}\log{(q-1)}-\frac{1}{4}\log{q}}, \\
    B(K)&=\frac{0.96K\log_2{q}}{\log K+\log\log q}.
\end{align*} Suppose that $\omega_0\geq (B\circ K \circ r)(\omega_0)$, $s(\omega_0)\geq 15$, $(K\circ r)(\omega_0)\geq 7$ and $(K\circ R)(\omega_0)\geq 3$. If we find that $\delta(\omega_0+1)>m$, then $\omega_0+1\geq (B\circ K\circ r)(\omega_0+1)>(B\circ K \circ R)(\omega_0+1)$, which implies the existence of a primitive element in $\cal{G}_\cal{A}$ for $\omega=\omega_0+1$. Additionally, $\delta$ is a decreasing function of $\omega$, so by induction, if $\delta(\omega_1)>m$ for some $\omega_1> \omega_0$ then there is a primitive element in $\cal{G}_\cal{A}$ whenever $\omega$ is between $\omega_0+1$ and $\omega_1$.
\end{lemma}
\begin{proof}
The proof is, \textit{mutatis mutandis}, the same as that of Lemma \ref{checkManyOmegasq4}. Note that the required lower bounds on $s$, $(K\circ r)$ and $(K\circ R)$ have changed. The bound on $s$ can decrease to 15 because it is now sufficient to have $r(\omega_0+1)-r(\omega_0)\leq \frac{1}{7}$. The bounds on $(K\circ r)$ and $(K\circ R)$ change because $B(K)$ is now strictly increasing if $K>\frac{e}{\ln{3}}\approx 2.474$ and is concave down if $K>\frac{e^2}{\ln{3}}\approx 6.726$. 
\end{proof}

\begin{prop}
    Let $q=3$ and suppose $r$ is an even integer greater than $2$. If $\cal{G}_\cal{A}$ does not contain a primitive element, then $\omega_r \leq 10^5$.
\end{prop}
\begin{proof}
We perform the \say{jumps} as before; in fact, one jump is enough. Let $\omega_0=10^5$, $m=0.01$ and $t=3\times 10^4$. Then $\delta(\omega_0)\approx 0.9024>m$. We have 
\begin{align*}
r(\omega_0) & \approx 41621.7 \\
s(\omega_0) & = 7 \times 10^4 \geq 15 \\
(K\circ r)(\omega_0) & \approx 578718.6\geq 7 \\
R(\omega_0) & \approx 41612.5 \\
(K\circ R)(\omega_0) & \approx 578590.5\geq 3 \\
\omega_0\geq (B\circ K\circ r)(\omega_0) & \approx 65897.1 .
\end{align*}
We need to find $\omega_1$ such that $\delta(\omega_1)>m$, that is, $$\sum_{i=1}^{\omega_1}(1/p_i)<0.99+\sum_{i=1}^{t}(1/p_i).$$ The right hand side is smaller than 3.798. Letting $n$ be the $\omega_1^\text{th}$ prime, we find that the value $n=1.143\times 10^{13}$ satisfies \eqref{eq:MertensCheck}. Then $\pi(n)>3.864\times 10^{11}$, so it is valid to pick $\omega_1=3.864\times 10^{11}$. It follows that $\cal{G}_\cal{A}$ contains a primitive element whenever $r$ is even and $10^5+1\leq \omega(3^r-1)\leq 3.864\times 10^{11}$. This, with Lemma \ref{sillyUpperBoundq3}, proves the result.
\end{proof}

\section{Checking small values of \texorpdfstring{$\omega$}{ω} and \texorpdfstring{$r$}{r}} \label{sec:ExplicSpec}
In what follows, we fix a value of $q$, then say that a value $\omega$ has been eliminated if we have proved $\cG_{\cA}$ contains a primitive element for any $r$ such that $\omega(q^r-1)=\omega$. For a particular $q$, we now check values of $\omega<N$ using Propositions \ref{checkOmega} and \ref{checkOmegaq3}. For each $\omega$, we must pick a value for $t$. In practice, one can usually use the same value of $t$ to eliminate many values of $\omega$.

For instance, when $q=5$, the value $t=33$ can be used to eliminate all values of $\omega$ in the range $[116,61366]$. Putting $t=4$ is enough to show the same for $[46,115]$. The value $t=3$ does not eliminate any smaller values of $\omega$, so we stop there. Using the result in Table \ref{table:naivePrimesLimit}, we have eliminated every integer value of $\omega$ greater than $45$. However, we are actually interested in eliminating values of $r$. Referring to the proof of Proposition \ref{checkOmega} (or Proposition \ref{checkOmegaq3} if $q=3$), we fix $t$, and pick $\omega=\omega_0 \geq t$. Then the condition $r>(K\circ R)(\omega_0)$ is enough to guarantee that $\cal{G}_\cal{A}$ will contain a primitive element. But $(K\circ R)$ is an increasing function of $\omega$, so 
$$r>(K\circ R)(\omega_0)\text{ implies } r>(K\circ R)(\omega) $$
for any $\omega$ such that $t\leq \omega \leq \omega_0$. Returning to our example of $q=5$, we select $t=4$ and $\omega_0=45$. Then $r>(K\circ R)(45)\approx103.639$ will guarantee that $\cal{G}_\cal{A}$ contains a primitive element. If $\cal{G}_\cal{A}$ has no primitive element, then we have $\omega \leq 45$, but when $4\leq \omega \leq 45$ we can also be sure that $r\leq 103$. Repeating with smaller values of $t$, we can obtain upper bounds that are valid when $\omega$ is as small as $1$. 

Table \ref{table:killingOmegas} shows the results of these computations for other values of $q$. The remaining values of $r$ are small enough that it is feasible to find and use the actual prime factorisation of $q^r-1$.

\begin{table}[!h]
    \centering
    \begin{tabular}{|c|c|c|}
    \hline
    $q$ & Values of $\omega$ eliminated & Values of $r$ eliminated \\
    \hline
    9 & $\omega \geq 27$ & $r \geq 39$ \\
    \hline
    8 & $\omega \geq 28$ & $r \geq 44$ \\
    \hline
    7 & $\omega \geq 31$ & $r \geq 52$ \\
    \hline
    5 & $\omega \geq 46$ & $r \geq 104$ \\
    \hline
    4 & $\omega \geq 120$ & $r \geq 391$ \\
    \hline
    3 & $\omega \geq 73$ & Even $r \geq 276$ \\
    \hline
    \end{tabular}
    \caption{The result of applying Propositions \ref{checkOmega} and \ref{checkOmegaq3}.}
    \label{table:killingOmegas}
\end{table}

We now employ Proposition \ref{hypersieve} directly. We fix a pair $(q,r)$ that we need to check, then find the distinct prime factors of $q^r-1$ (of which there are $\omega_r$). For each $s$ such that $1\leq s \leq \omega_r$, we take the largest $s$ prime factors and use them as the $p_i$. We then check whether the relevant inequality in Proposition \ref{hypersieve} holds. Note that, if we use (\ref{eq:hypersieveInequality}), we will modify it as in Remark \ref{remarkRemoveCeiling}. If the (possibly modified) inequality holds for any $s$, we conclude that $\cal{G}_\cal{A}$ contains a primitive element. Tables \ref{table:killingRs} and \ref{table:killingRsCW} show the pairs $(q,r)$ that are eliminated via  \eqref{eq:hypersieveInequality} and \eqref{eq:hypersieveInequalityEven}, respectively. For $q\geq 5$, Mathematica readily determined the factorisation of each $q^r-1$ that was tested. For $q\in\{3,4\}$, the authors resorted to the online database \cite{cowNoise}. For each required value of $r$, the database provided the list of prime factors of $q^r-1$. Having eliminated many values of $r$, we now apply the following two criteria (used by Fernandes and Reis in \cite{FernandesReis2021}) to the remaining cases.
The first criterion is simple and can usually eliminate many small values of $r$ if $q\geq 7$.
\begin{prop}\label{FRcriterion1}
If $(q-1)^r>q^r - \varphi(q^r-1)$, then $\cal{G}_\cal{A}$ contains a primitive element.
\end{prop}
\begin{proof}
The result follows from the fact that there are $\varphi(q^r-1)$ primitive elements in $\F_{q^r}^*$ and $\vert \F_{q^r} \vert = q^r$. Indeed, if $\cal{G}_\cal{A}$ contained no primitive elements, then we must have $\vert \F_{q^r} \setminus \cal{G}_\cal{A} \vert \ge \varphi(q^r-1)$, a contradiction.
\end{proof}
\begin{table}[!h]
    \centering
    \begin{tabularx}{\linewidth}{|c|c|L|}
    \hline
    $q$& Values of $r$ that were checked & Values of $r$ that were eliminated\\
    \hline
    9 & $2\leq r\leq 38$ & 13, 16, 17 and 19-38 \\
    \hline
    8 & $2\leq r \leq 43$ & 13, 15, 17, 18, 19 and 21-43\\
    \hline 
    7 & $2\leq r \leq 51$ & 17, 19, 21, 22, 23 and 25-51 \\
    \hline
    5 & $2\leq r \leq 103$ & 29, 31, 33, 37, 39, 41, 43, 45, 46, 47 and 49-103 \\
    \hline
    4 & $2\leq r \leq 390$ & 89, 97, 101, 103, 107, 109, 113, 119, 121, 125, 127, 129, 131, 133, 137, 139, 141-143, 145-149, 151-153, 155, 157-161, 163-167, 169-179, 181-209, 211-390 \\
    \hline
    \end{tabularx}
    \caption{The pairs $(q,r)$ that were eliminated using (\ref{eq:hypersieveInequality}), making use of the actual factorisation of $q^r-1$.}
    \label{table:killingRs}
\end{table}

\begin{table}[!h]
    \centering
    \begin{tabularx}{\linewidth}{|c|c|L|}
    \hline
    $q$& Values of $r$ that were checked (even values only) & Values of $r$ that were eliminated\\
    \hline
    9 & $2\leq r\leq 18$ & 14, 16, 18 \\
    \hline
    8 & $2\leq r \leq 20$ & 14\\
    \hline 
    7 & $2\leq r \leq 24$ & 22, 24 \\
    \hline
    5 & $2\leq r \leq 48$ & 26, 28, all even numbers between 32 and 48 \\
    \hline
    4 & $2\leq r \leq 210$ & 34, 38, 40, all even numbers between 44 and 210 \\ 
    \hline
    3 & $4\leq r \leq 268$ & 86, 106, 110, 116, 118, all even numbers between 122 and 268 except 144 \\
    \hline
    \end{tabularx}
    \caption{The pairs $(q,r)$ that were eliminated using (\ref{eq:hypersieveInequalityEven}), making use of the actual factorisation of $q^r-1$.}
    \label{table:killingRsCW}
\end{table}
The second criterion (proved in \cite{FernandesReis2021}) eliminates some larger values of $r$ if $q\geq 7$. It involves a bound on character sums over an $\F_q$-affine subspace $\cA\subseteq \F_{q^n}$; the bound depends on the dimension of the subspace.
\begin{prop}\cite[Th. 3.2]{FernandesReis2021} \label{FRcriterion2}
If $(q-1)^r>\alpha(q,r)2^{\omega_r}$, where 
$$\alpha(q,r)=\sum_{i=0}^{r-1}\binom{r}{i}q^{\min\{i,r/2\}},$$ 
then $\cal{G}_\cal{A}$ contains a primitive element.
\end{prop}

For specific $q$ and $r$, it is easy to check whether either of these criteria are satisfied. Given $q$, we allow $r$ to range from 2 to the largest value not eliminated in the previous step. Table \ref{table:killingRsUsingFR} shows the results of applying the last two criteria. This also completes the proof of Theorem \ref{thm:mainHyperplaneResult}.

\begin{table}
    \centering
    \begin{tabular}{|c|c|c|}
    \hline
    $q$& Values of $r$ that were checked & Values of $r$ that were eliminated\\
    \hline
    9 & $2\leq r\leq 15$ & $r=2,3,4,5,7,11,13,14$\\
    \hline
    8 & $2\leq r \leq 20$ & $r=2,3,4,5,7,9,11,13,17,18,19$\\
    \hline 
    7 & $2\leq r \leq 20$ & $r=2,13,19$\\
    \hline
    5 & $2\leq r \leq 35$ & None\\
    \hline
    4 & $2\leq r \leq 135$ & $r=2$ \\
    \hline
    3 & $2\leq r \leq 120$ (even $r$ only), and 144 & None\\
    \hline
    \end{tabular}
    \caption{The pairs $(q,r)$ that were eliminated using Propositions \ref{FRcriterion1} and \ref{FRcriterion2}.}
    \label{table:killingRsUsingFR}
\end{table}

\section{Some genuine exceptions}\label{sec:Exceptions}
In Section 1, we observed that the possible exceptions listed in Theorem \ref{thm:mainHyperplaneResult} might not all be genuine. We now prove that three of those possible exceptions are genuine. Specifically, we have the following result.

\begin{prop}
    When $(q,r)$ is either $(3,2)$, $(5,2)$ or $(3,3)$, it is possible to select hyperplanes $C_1,\ldots,C_r$ such that every primitive element lies on at least one of the $C_i$. That is, such that $\cG_{\cA}$ does \textbf{not} contain a primitive element.
\end{prop}
\begin{proof}
For $(q,r)=(3,2)$, we have the finite field extension $\F_{3^2}/\F_3$. To describe it we use the isomorphism $\F_9\cong \F_3[x]/(x^2-2)$. The unit group has order $3^2-1=8$, so there are $\varphi(8)=4$ primitive elements.

One can verify that $\alpha=1+x$ is a primitive element (it is sufficient to check that $\alpha^4\neq 1$), so the set of primitive elements is 
$$\{\alpha=1+x,\alpha^3=1+2x,\alpha^5=2+2x,\alpha^7=2+x\}.$$
Select the hyperplanes \begin{align*}
    H_1&=\{a_1(1+x)\mid a_1\in \F_3\} \\
    H_2&=\{b_1(1+2x)\mid b_1\in \F_3\}.
\end{align*} It is clear that $\alpha$ and $\alpha^5$ lie on $H_1$, and that $\alpha^3$ and $\alpha^7$ lie on $H_2$. The set of elements not on the hyperplanes, $\cal{G}_\cal{A}$, thus contains no primitive elements.

For $(q,r)=(5,2)$, we use $\F_{25}\cong \F_5[x]/(x^2-2)$. $\alpha=2+x$ is a primitive element, so the set of primitive elements is 
$$\{2+x,2+4x,4+2x,4+3x,3+4x,3+x,1+3x,1+2x\}.$$ 
We choose hyperplanes\begin{align*}
    H_1&=\{a_1(2+x)\mid a_1\in \F_5\}\\
    H_2&=\{b_1(3+x)\mid b_1\in \F_5\};
\end{align*} again finding that every primitive element lies on one of them.

For $(q,r)=(3,3)$, write $F_{27}\cong \F_3[x]/(x^3+2x+2)$. $2+2x$ is a primitive element; the set of primitive elements is 
\begin{align*}
\{2+2x,1+2x,2x^2,1+x^2,2x,2+x+x^2,2+x+2x^2, & \\
1+2x^2,2+2x+2x^2, 2+2x+x^2,2x+2x^2,x+2x^2\}. &
\end{align*}
Choose hyperplanes
\begin{align*}
    H_1&=\{2x^2+a_1(1)+a_2(x)\mid a_1,a_2\in \F_3\}\\
    H_2&=\{2x+b_1(1)+b_2(x^2)\mid b_1,b_2\in \F_3\}\\
    H_3&=\{x^2+1+c_1(x^2)+c_2(1+x)\mid c_1,c_2\in\F_3\}.
\end{align*} Every primitive element is on at least one of them.    
\end{proof}

\bibliographystyle{plain}
\bibliography{refs}

\end{document}